\documentclass[12pt,twoside]{amsart}
\usepackage{amsmath}
\usepackage{amsthm}
\usepackage{amsfonts}
\usepackage{amssymb}
\usepackage{latexsym}
\usepackage{mathrsfs}
\usepackage{amsmath}
\usepackage{amsthm}
\usepackage{amsfonts}
\usepackage{amssymb}
\usepackage{latexsym}
\usepackage{geometry}
\usepackage{dsfont}
\usepackage[dvips]{graphicx}
\usepackage{color}
\usepackage[all]{xy}

\usepackage{amsthm,amsmath,amssymb}\usepackage{mathrsfs}

\date{}
\allowdisplaybreaks[4] \footskip=15pt
\renewcommand{\uppercasenonmath}[1]{}

\usepackage{graphicx,amssymb}
\usepackage[all]{xy}
\usepackage{amsmath}

\allowdisplaybreaks
\usepackage{amsthm}
\usepackage{color}

\theoremstyle{plain}
\newtheorem{theorem}{Theorem}[section]
\newtheorem{proposition}[theorem]{Proposition}
\newtheorem{lemma}[theorem]{Lemma}
\newtheorem{corollary}[theorem]{Corollary}
\theoremstyle{definition}

\newtheorem{definition}[theorem]{Definition}

\theoremstyle{definition}

\theoremstyle{remark}
\newtheorem{remark}[theorem]{Remark}

\newcommand{\Lsr}{\mathcal{L}}

\def\bc{\begin{center}}
\def\ec{\end{center}}

\def\im{{\rm Im}}

\def\Ext{{\rm Ext}}
\def\Tor{{\rm Tor}}

\def\pd{{\rm pd}}

\def\fd{{\rm fd}}

\def\FPD{{\rm FPD}}

\def\Hom{{\rm Hom}}

\def\fd{{\rm fd}}

\def\Hom{{\rm Hom}}
\def\Ext{{\rm Ext}}
\def\Tor{{\rm Tor}}

\def\im{{\rm im}}

\def\Nil{{\rm Nil}}

\def\fd{{\rm fd}}

\def\Hom{{\rm Hom}}
\def\Ext{{\rm Ext}}
\def\Tor{{\rm Tor}}

\def\im{{\rm im}}

\def\Nil{{\rm Nil}}

\def\T{{\rm T}}

\def\Krull{{\rm Krull}}
\def\Nil{{\rm Nil}}
\def\Prufer{{\rm Pr\"{u}fer}}

\def\DQ{{\rm DQ}}
\def\fPD{{\rm fPD}}
\def\m{{\frak m}}
\def\T{{\rm T}}

\begin{document}
\begin{center}
{\large  \bf  Semi-regular flat modules over strong Pr\"{u}fer rings}

\vspace{0.5cm}   Xiaolei Zhang$^{a}$, Guocheng Dai$^b$, Xuelian Xiao$^c$, Wei Qi$^b$\\

{\footnotesize a.\ Department of Basic Courses, Chengdu Aeronautic Polytechnic, Chengdu 610100, China\\
b.\ School of Mathematical Sciences, Sichuan Normal University, Chengdu 610068, China\\
c.\ School of Mathematics, ABa Teachers University, Wenchuan 623002, China\\

E-mail: zxlrghj@163.com\\}
\end{center}

\bigskip
\centerline { \bf  Abstract}  We first introduce and study the notion of semi-regular flat modules, and then show that a ring $R$ is a strong \Prufer\ ring if and only if every submodule of a semi-regular flat $R$-module is semi-regular flat, if and only if every  ideal of $R$ is semi-regular flat, if and only if every $R$-module has a surjective semi-regular flat (pre)envelope.
\bigskip
\leftskip10truemm \rightskip10truemm \noindent
\vbox to 0.3cm{}

\leftskip0truemm \rightskip0truemm
\bigskip
\section{introduction}

In this paper, we always assume $R$ is a commutative ring with identity and $\T(R)$ is the total ring of fractions of $R$. Following from \cite{fk16}, an ideal $I$ of $R$ is said to be dense if $(0:_RI):=\{r\in R|Ir=0\}$ is $0$, or be semi-regular if it contains a finitely generated dense sub-ideal, or be regular if it contains a regular element. Let $I$ be an ideal of $R$. Denote by $I^{-1}=\{z\in \T(R)|Iz\subseteq R\}$. If an ideal $I$ of $R$ satisfies $II^{-1}=R$, then $I$ is said to be an invertible ideal.

Early in 1932, \Prufer\ \cite{P32} introduced  integral domains over which all finitely generated non-zero ideals are invertible, and then they are called  \Prufer\ domains by \Krull\ \cite{K37}. Many algebraists have done a lot of work on \Prufer\ domains. For the system summaries of \Prufer\ domains, one can refer to \cite{FHP97}. Since \Prufer\ domains are of great importance to the study of integral domains, many scholars generalized the notion of integral domains to these of commutative rings with zero-divisors. In 1967, Butts and Smith \cite{BS67} introduced the notion of \Prufer\ rings over which every finitely generated regular ideal is invertible. And then, Griffin \cite{G69} characterized \Prufer\ rings utilizing multiplicative ideal theory. Recently, Xiao et al. \cite{xwl20} characterized \Prufer\ rings by module-theoretic viewpoint.

Since the notion of \Prufer\ rings is very simple, it is very hard to delve deeper (all total rings of quotients are \Prufer\ rings). For better understanding \Prufer\ rings, Anderson et al. \cite{AA85} introduced the notion of strong \Prufer\ rings, over which every finitely generated semi-regular ideal is locally principal, and they showed that a ring $R$ is strong \Prufer\ if and only if its Nagata ring $R(x)$ is a \Prufer\ ring. In 1987, Anderson et al. \cite{AP87} characterized strong \Prufer\ rings by lattice-theoretic viewpoint, i.e., a ring $R$ is strong \Prufer\ if and only if the lattice $\Lsr_{sr}(R)$ consisting of all finitely generated semi-regular ideals is a distributive lattice. In 1993, Lucas \cite{L93} proved that a ring $R$ is strong \Prufer\ if and only if $R$ is a \Prufer\ ring and $\T(R)$ is a strong \Prufer\ ring. The small finitistic dimensions \fPD($R$) ($\fPD(R) = \sup\{\pd_RM \mid M\ \mbox{is super finitely presented}, \ \pd_RM<\infty\}$) of a strong \Prufer\ ring $R$ is also very attractive. In 2020, Wang et al. \cite{fkxs20} showed that a ring $R$ satisfies $\fPD(R)=0$ if and only if $R$ is a $\DQ$ ring, i.e., the only finitely generated semi-regular ideal of $R$ is $R$ itself. Subsequently, Wang et al. \cite{wzq20} showed that a connect strong \Prufer\ ring has $\fPD(R)\leq 1$. Recently, The first author of the paper and Wang showed that every strong \Prufer\ ring has $\fPD(R)\leq 1$, and obtained examples of  total rings of quotients $R$ with $\fPD(R)=n$ for any $n\in\mathbb{N}$.

The main motivation of this paper is to give some module-theoretic and homology-theoretic characterizations of strong \Prufer\ rings. First, we introduce the notions of semi-regular flat modules, semi-regular coflat modules and semi-regular cotorsion modules, and show that the classes of semi-regular flat modules and semi-regular cotorsion modules constitute a perfect cotorsion pair. Then we give some characterizations of $\DQ$ rings and strong \Prufer\ rings. More precisely, we show that a ring $R$ is a  $\DQ$ ring if and only if every $R$-module is semi-regular flat, if and only if every $R$-module is semi-regular coflat, if and only if every semi-regular cotorsion module is injective; we also show that a ring $R$ is strong \Prufer\ if and only if every submodule of a semi-regular flat module is semi-regular flat modules, if and only if every finitely generated ideal of $R$ is semi-regular flat. Finally, we introduce and characterize the notion of semi-regular coherent rings. We also show that a ring $R$ is strong \Prufer\ if and only if every $R$-module has a surjective semi-regular flat envelope, if and only if every $R$-module has a surjective semi-regular flat preenvelope.

\section{Semi-regular flat modules}

Recall from \cite{xwl20}, an $R$-module $M$ is called a regular flat module if $\Tor_1^R(R/I,M)=0$ for any finitely generated regular ideal $I$. Obviously, every flat module is regular flat. The notion of regular flat modules is used to characterize the total rings of quotients (see \cite{xwl20}). For studying $\DQ$-rings and strong \Prufer\ rings, we introduced the notion of semi-regular flat modules.

\begin{definition}\label{sr-flat }
An $R$-module $M$ is said to be a semi-regular flat module if, for any finitely generated semi-regular ideal $I$, we have $\Tor_1^R(R/I,M)=0$. The class of all semi-regular flat modules is denoted by $\mathcal{F}_{sr}$.
\end{definition}

Obviously, any flat module is semi-regular flat, and any semi-regular flat module is regular flat. Dually, we can give the notion of semi-regular coflat modules.
\begin{definition}\label{sr-flat }
An $R$-module $M$ is said to be a semi-regular coflat module if, for any finitely generated semi-regular ideal $I$, we have $\Ext^1_R(R/I,M)=0$.
\end{definition}

\begin{lemma}\label{101}
 Let $M$ be an $R$-module. Then the following statements are equivalent:
\begin{enumerate}
 \item $M$ is a semi-regular flat module;
 \item for any finitely generated semi-regular ideal $I$, the natural homomorphism $I\otimes M\rightarrow R\otimes M$ is a monomorphism;
 \item for any finitely generated semi-regular ideal $I$, the natural homomorphism $\sigma_I: I\otimes M\rightarrow IM$ is an isomorphism;
 \item for any injective module $E$, $\Hom_R(M,E)$ is a semi-regular coflat module;
 \item if $E$ is injective cogenerator, then $\Hom_R(M,E)$ is a semi-regular coflat module.
\end{enumerate}
\end{lemma}
\begin{proof} $(1)\Leftrightarrow (2)$: Let $I$ be a finitely generated semi-regular ideal. Then we have a long exact sequence:  $$0\rightarrow \Tor_1^R(R/I,M)\rightarrow I\otimes M\rightarrow R\otimes M\rightarrow R/I\otimes M\rightarrow 0.$$ Consequently, $\Tor_1^R(R/I,M)=0$ if and only if
$I\otimes M\rightarrow R\rightarrow R\otimes M$ is a monomorphism.

$(2)\Rightarrow (3)$: Let $I$ be a finitely generated semi-regular ideal. Then we have the following commutative diagram:
$$\xymatrix{
 0 \ar[r]^{} &I\otimes M \ar[d]_{\sigma_I}\ar[r]^{} & R\otimes_RM \ar[d]^{\cong}_{\sigma_R}\\
 0 \ar[r]^{} & IM \ar[r]^{} &M.
}$$
Then $\sigma_I$ is a monomorphism. Since the multiplicative map $\sigma_I$ is an epimorphism, $\sigma_I$ is actually an isomorphism.

$(3)\Rightarrow (1)$: Let $I$ be a finitely generated semi-regular ideal. Then we have a long exact sequence:
$$\xymatrix{
 0\ar[r]^{} & \Tor_1^R(R/I,M) \ar[r]^{} &I M \ar[r]^{f} & M.
}$$
Since $f$ is a  natural embedding map, we have $\Tor_1^R(R/I,M)=0$.

$(1)\Rightarrow (4)$: Let $I$ be a finitely generated semi-regular ideal and $E$  an injective module. Then $\Ext_R^1(R/I,\Hom_R(M,E))\cong\Hom_R(\Tor_1^R(R/I,M),E)$. Since $M$ is a semi-regular flat module, then $\Tor_1^R(R/I,M)=0$. Thus $\Ext_R^1(R/I,\Hom_R(M,E))=0$, So $\Hom_R(M,E)$ is a semi-regular coflat module.

$(4)\Rightarrow (5)$: Trivial.

$(5)\Rightarrow (1)$: Let $I$ be a finitely generated semi-regular ideal and $E$  an injective cogenerator. Since $\Hom_R(M,E)$ is a regular coflat module and $$\Ext_R^1(R/I,\Hom_R(M,E))\cong\Hom_R(\Tor_1^R(R/I,M),E),$$ we have $\Hom_R(\Tor_1^R(R/I,M),E)=0$. Since $E$ is an injective cogenerator, $\Tor_1^R(R/I,M)=0$. So $M$ is a semi-regular flat module.
\end{proof}

\begin{lemma}\label{102}
Let $R$ be a ring. Then the class $\mathcal{F}_{sr}$ of all semi-regular flat modules is closed under direct limits, pure submodules and pure quotients.
\end{lemma}
\begin{proof}
For the  direct limits, suppose $\{M_i\}_{i\in\Gamma}$ is a direct system consisting of semi-regular flat modules. Then, for any finitely generated semi-regular ideal $I$, we have $\Tor_1^R(R/I,\lim\limits_{\longrightarrow } M_i)=\lim\limits_{\longrightarrow}\Tor_1^R(R/I,M_i)=0$. So $\lim\limits_{\longrightarrow }M_i$ is a semi-regular flat module.

For pure submodules and pure quotients, let $I$ be a finitely generated semi-regular ideal. Suppose $0\rightarrow M\rightarrow N\rightarrow L\rightarrow 0$ is a pure exact sequence. We have the following commutative diagram with rows exact :
$$\xymatrix{
 0 \ar[r]^{} & M\otimes_R I \ar[d]_{f}\ar[r]^{} & N\otimes_R I \ar@{>->}[d]_{}\ar[r]^{} & L\otimes_R I \ar[d]_{g}\ar[r]^{} & 0\\
 0 \ar[r]^{} & M\otimes_R R \ar@{->>}[d]_{}\ar[r]^{} & N\otimes_R R \ar@{->>}[d]_{}\ar[r]^{} & L\otimes_R R \ar@{->>}[d]_{}\ar[r]^{} & 0\\
 0 \ar[r]^{} & M\otimes_R R/I \ar[r]^{} & N\otimes_R R/I \ar[r]^{} & L\otimes_R R/I \ar[r]^{} & 0\\}$$
By the Snake Lemma, the natural homomorphism $f: M\otimes_R I \rightarrow M\otimes_R R$ and $g: L\otimes_R I \rightarrow L\otimes_R R$ are all monomorphisms. Consequently, $M$ and $L$ are all semi-regular flat.
\end{proof}

\begin{definition}\label{sr-cotorsion }
An $R$-module $N$ is said to be semi-regular cotorsion if, for any semi-regular flat module $M$, we have $\Ext^1_R(M,N)=0$. The class of all semi-regular cotorsion modules is denoted by $\mathcal{C}_{sr}$.
\end{definition}
Obviously, any injective module is semi-regular cotorsion, and any semi-regular cotorsion module is cotorsion. It is well-known that the classes of all flat modules and all cotorsion modules constitute  a perfect  cotorsion pair (see \cite{BBE01}). Now, we show that the class of all semi-regular flat modules and all semi-regular cotorsion modules also  constitute  a perfect  cotorsion pair.
\begin{theorem}\label{204}
Let $R$ be a ring. Then  $(\mathcal{F}_{sr},\mathcal{C}_{sr})$ is a perfect cotorsion pair. Consequently, the class $\mathcal{F}_{sr}$ of all  semi-regular flat modules is covering and the class $\mathcal{C}_{sr}$ of all  semi-regular cotorsion modules is enveloping.
\end{theorem}
\begin{proof} Obvioulsy, $R$ itself is a semi-regular flat module, and the class $\mathcal{F}_{sr}$ is closed under direct summands and extensions. By Lemma \ref{102} and  \cite[Theorem 3.4]{HJ08}, we have
$(\mathcal{F}_{sr},\mathcal{C}_{sr})$ is a perfect cotorsion pair. So  the class $\mathcal{F}_{sr}$ of all  semi-regular flat modules is covering and the class $\mathcal{C}_{sr}$ of all  semi-regular cotorsion modules is enveloping.
\end{proof}
\begin{proposition}\label{209}
Let $R$ be a ring. Then the following statements are equivalent:
\begin{enumerate}
 \item $M$ is a semi-regular cotorsion module;
 \item  $\Hom_R(F,M)$ is semi-regular cotorsion  for any flat module $F$;
 \item $\Hom_R(P,M)$ is semi-regular cotorsion for any projective module $P$.
\end{enumerate}
\end{proposition}
\begin{proof}
$(1)\Rightarrow (2)$: Let $N$ be a semi-regular flat module and $F$  a flat module. There is a short exact sequence
 $0\rightarrow K\rightarrow P\rightarrow N\rightarrow 0$, where $P$ is  projective. So there is an exact sequence $0\rightarrow K\otimes_R F\rightarrow P\otimes_R F\rightarrow N\otimes_R F\rightarrow 0$.
Note that $\Tor^R_1(R/I,N\otimes_R F)=\Tor^R_1(R/I,N)\otimes_R F=0$ for any finitely generated semi-regular ideal $I$, thus $N\otimes_R F$ is a semi-regular flat $R$-module.
Since $$\Hom_R(P\otimes_R F,M)\rightarrow \Hom_R(K\otimes_R F,M)\rightarrow \Ext_R^1(N\otimes_R F,M)=0$$ is exact, there is a short exact sequence $$\Hom_R(P,\Hom_R(F,M))\rightarrow \Hom_R(K,\Hom_R(F,M))\rightarrow 0.$$
On the other hand, the sequence $$(P,(F,M))\rightarrow (K,(F,M))\rightarrow \Ext_R^1(N,(F,M))\rightarrow \Ext_R^1(P,(F,M))=0$$ is exact (Use $(-,-)$ to instead of $\Hom_R(-,-)$). Consequently, $$\Ext_R^1(N,\Hom_R(F,M))=0.$$ Thus $\Hom_R(F,M)$ is a semi-regular cotorsion $R$-module.

$(2)\Rightarrow (3)$: Trivial.

$(3)\Rightarrow (1)$: Set $P=R$, then the result holds.
\end{proof}

\section{The homology theory of semi-regular flat modules}
Recall from \cite[Porposition 2.2]{fkxs20}, a ring $R$ is called a  $\DQ$ ring  if the only  finitely generated semi-regular ideal of $R$ is $R$ itself. It is well-known that a  ring $R$ is a von Neumann regular ring if and only if  any $R$-module is flat.
\begin{theorem}\label{201}
Let $R$ be a ring. Then the following statements are equivalent:
\begin{enumerate}
 \item $R$ is a $\DQ$ ring;
 \item \fPD$(R)=0$;
 \item every $R$-module is semi-regular flat;
 \item every $R$-module is semi-regular coflat;
 \item every semi-regular cotorsion module is injective;
 \item for every finitely generated semi-regular ideal $I$ and any finitely generated ideal $J$, we have $I\cap J=IJ$;
 \item for every finitely generated semi-regular ideal $I$, $R/I$ is flat module;
 \item for every finitely generated semi-regular ideal $I$ and $a\in I$, there is $c\in I$ such that $(1-c)a=0$.
\end{enumerate}
\end{theorem}
\begin{proof}$(1)\Rightarrow (3)$, $(1)\Rightarrow (4)$ and $(1)\Rightarrow (7)$: Trivial.

$(1)\Leftrightarrow (2)$: See \cite[Proposition 2.2]{fkxs20}. $(3)\Leftrightarrow (5)$: It follows form Theorem \ref{204}. $(6)\Leftrightarrow (7)\Leftrightarrow (8)$: See \cite[Theorem 1.2.15]{g}.

$(3)\Rightarrow (6)$: Let $I$ be a finitely generated semi-regular ideal of $R$ and $J$ a finitely generated ideal  of $R$. Then $R/J$ is a semi-regular flat module. So $\Tor_1^R(R/I,R/J)=0$, that is, $I\cap J=IJ$ (See \cite[Exercise 3.20]{fk16}).

$(7)\Rightarrow (1)$:  Let $I$ be a finitely generated semi-regular ideal of $R$  of $R$ and $M$  an $R$-module. Then we have $\Tor_1^R(R/I, M)=0$.

$(8)\Rightarrow (1)$: Suppose $I=\langle a_1,...,a_n\rangle$ is finitely generated semi-regular ideal. Then, for any $i=1,...,n$, there exists $c_i\in I$ such that $(1-c_i)a_i=0$. Set $c=\prod_{i=1}^n(1-c_i)$. Then $ca_i=0$($i=1,...,n$). Thus $c\in (0:_RI)=0$. Note that $1-c\in I$. Thus $1\in I$ and $I=R$.

$(4)\Rightarrow (1)$:  Let $I$ be a finitely generated semi-regular ideal of $R$.  Then $R/I$ is a projective module by (4). Thus $I$ is finitely generated idempotent ideal. By \cite[Chapter I, Theorem 1.10]{FS01}, $I=\langle e\rangle$ where $e$ is an  idempotent. Then $1-e\in (0:_RI)=0$. So $I=R$.
\end{proof}

\begin{remark}\label{sr-flat}
Now we give examples of non-flat semi-regular flat modules, and non-semi-regular  regular flat modules.
\begin{enumerate}
 \item Suppose the non-field ring $R$ is a Noetherian local ring with Krull dimension equal to $0$ (Take $R=\mathbb{Z}_{p^n}$ for example). By \cite{AB58}, we have \fPD$(R)=0 $. Since $R$ is a local ring  which is not a field, then $R$  is not a von Neumann regular ring. By Theorem \ref{201}, there exists   a semi-regular flat $R$-module which is not flat.

 \item Following from \cite[Theorem 2.13]{xwl20}, $R$ is a total ring of quotient if and only if any $R$-module  is regular flat. Wang \cite{wzcc20} give an example of total rings of quotients $R$ with \fPD$(R)>0$ . Then there exist regular flat modules which are not semi-regular flat by Theorem \ref{201}.
 \end{enumerate}
\end{remark}

\begin{corollary}\label{202}
Let $R$ be a ring. Then the following statements are equivalent:
\begin{enumerate}
 \item $R$ is a semi-simple ring;
 \item any semi-regular flat $R$-module is projective;
 \item any $R$-module is semi-regular cotorsion.
\end{enumerate}
\end{corollary}
\begin{proof} $(1)\Rightarrow (2)$ and $(1)\Rightarrow (3)$: Trivial.

$(2)\Rightarrow (1)$: Since any flat module is semi-regular flat, any flat module is projective by $(2)$. So $R$ is a perfect ring, that is, $\FPD(R)=0$. So $\fPD(R)=0$. By Theorem \ref{201}, any $R$-module  is semi-regular flat module, so is projective by $(2)$ again. Consequently, $R$ is a semi-simple ring.

 $(2)\Leftrightarrow (3)$: It follows by Theorem \ref{204}.
\end{proof}

\begin{theorem}\label{203}
 Let $R$ be a ring. Then the following statements are equivalent:
\begin{enumerate}
 \item $R$ is a strong \Prufer\ ring;
 \item any submodule of a semi-regular flat $R$-module  is semi-regular flat;
 \item any submodule of a flat $R$-module  is semi-regular flat;
 \item any ideal of $R$ is semi-regular flat;
 \item any finitely generated ideal of $R$ is semi-regular flat;
 \item any finitely generated semi-regular ideal  of $R$ is flat;
 \item any finitely generated semi-regular ideal  of $R$ is projective;
 \item any quotient of a semi-regular coflat $R$-module  is semi-regular coflat;
 \item any $h$-divisible $R$-module is semi-regular coflat.
\end{enumerate}
\end{theorem}
\begin{proof} $(2)\Rightarrow (3)\Rightarrow (4)\Rightarrow (5)$, $(7)\Rightarrow (6)$ and $(8)\Rightarrow (9)$: Trivial.

$(5)\Leftrightarrow (6)$: Let $I$ be a finitely generated semi-regular ideal of $R$ and $J$ a finitely generated ideal  of $R$. Then we have  $\Tor_1^R(R/J, I)\cong\Tor_2^R(R/I, R/J)\cong\Tor_1^R(R/I, J)$.
Consequently, $J$ is semi-regular flat if and only if $I$ is flat.

$(6)\Rightarrow (1)$: Let $I$ be a finitely generated semi-regular ideal of $R$ and $\m$  a maximal ideal of  $R$. Then $I_{\m}$ is finitely generated flat $R_{\m}$-ideal. By \cite[Lemma 4.2.1]{g} and \cite[Theorem 2.5]{M89}, we have $I_{\m}$ is a free $R_{\m}$-ideal. So the rank of  $I_{\m}$ is at most $1$. Consequently, $I_{\m}$ is a principal ideal of $R_{\m}$.

$(1)\Rightarrow (6)$: Let $I$ be a finitely generated semi-regular ideal of $R$ and $\m$  a maximal ideal of  $R$.  Then $I_{\m}$ is a principal $R_{\m}$-ideal. Suppose $I_{\m}=\langle \frac{x}{s}\rangle$. Then  $(0:_{R_{\m}}\frac{x}{s})=(0:_{R_{\m}}I_{\m})=(0:_RI)_{\m}=0$ by \cite[Exercise 1.72] {fk16}.  Thus $\frac{x}{s}$ is regular element. So $I_{\m}\cong R_{\m}$. Consequently, $I$ is a flat $R$-ideal.

$(6)\Rightarrow (4)$:  Let $I$ be a finitely generated semi-regular ideal of $R$  and $K$  an ideal of $R$ . Then we have $\Tor_1^R(R/I, K)\cong\Tor_1^R(R/K, I)=0$ by (6). So $K$ is  semi-regular flat.

$(6)\Rightarrow (2)$: Let $M$ be a semi-regular flat module and $N$  a submodule of $M$. Suppose $I$ is a finitely generated semi-regular ideal, then $I$ is a flat ideal. Thus $\fd_R(R/I)\leq 1$. Consider the exact sequence $$\Tor_2^R(R/I,M/N)\rightarrow \Tor_1^R(R/I,N)\rightarrow \Tor_1^R(R/I,M).$$ Since $\Tor_2^R(R/I,M/N)=\Tor_1^R(R/I,M)=0$, we have $\Tor_1^R(R/I,N)=0$. So $N$ is a semi-regular flat module.

$(6)\Rightarrow (7)$: It follows from \cite[Corollary 3.1]{V69}.

$(9)\Rightarrow (7)$:  Let $I$ be a finitely generated semi-regular ideal of $R$  and $M$ an $R$-module. Consider the exact sequence $0\rightarrow M\rightarrow E\rightarrow N\rightarrow 0$, where $E$ is an injective module. Then $N$ is an $h$-divisible module. Consequently $\Ext^1_R(I,M)\cong\Ext^2_R(R/I,M)\cong\Ext^1_R(R/I,N)=0$. So $I$ is a projective ideal of $R$.

$(7)\Rightarrow (8)$: Let $I$ be a finitely generated semi-regular ideal of $R$. Then $I$ is a projective ideal. Consequently, $\pd_R(R/I)\leq 1$. Suppose $0\rightarrow L\rightarrow M\rightarrow N\rightarrow 0$ is a short exact sequence, where $M$ is a semi-regular coflat module. Then we have an exact sequence  $\Ext^1_R(R/I,M)\rightarrow \Ext^1_R(R/I,N)\rightarrow \Ext^2_R(R/I,L)$. Since $M$ is a semi-regular coflat module, $\Ext^1_R(R/I,M)=0$. Since $\pd_R(R/I)\leq 1$, we have $\Ext^2_R(R/I,L)=0$. Thus $\Ext^1_R(R/I,N)=0$. So $N$ is a semi-regular coflat module.
\end{proof}

\section{Semi-regular coherent rings}

Recall from \cite{g} that a ring $R$ is called a coherent ring if any finitely generated ideal  is finitely presented. Some important rings are all coherent, such as Noetherian rings, \Prufer\ domains. However, strong \Prufer\ rings are not coherent in general. So we introduce the notion of semi-regular coherent rings and give some new characterization of strong  \Prufer\ rings using semi-regular flat (pre)envelopes.

\begin{definition}\label{w-phi-coherent }
A ring $R$ is called a semi-regular coherent ring if any finitely generated semi-regular ideal $I$ is finitely presented.
\end{definition}

Let $R$  be non-semi-hereditary ring with weak global dimension at most $1$. Then $R$ is strong \Prufer\ ring, but $R$ is not coherent (See \cite[Corollary 4.2.19]{g}). The following result shows that any strong \Prufer\ ring is semi-regular coherent.
\begin{proposition}\label{sp-sc}
Suppose $R$ is a strong \Prufer\ ring, then $R$ is semi-regular coherent.
\end{proposition}
\begin{proof} Suppose $I$ is a finitely generated semi-regular ideal of $R$. Then, by Theorem \ref{203}, $I$ is a projective ideal of $R$, and thus is a finitely presented ideal of $R$.
\end{proof}

Some examples of non-integral domains are constructed by idealization $R(+)M$, where $M$ is an $R$-module (See \cite{H88}). Set $R(+)M$ isomorphic to $R\oplus M$ as $R$-modules. Define
\begin{enumerate}
 \item ($r,m$)+($s,n$)=($r+s,m+n$),
 \item ($r,m$)($s,n$)=($rs,sm+rn$).
\end{enumerate}
Under these operations, $R(+)M$ is a commutative ring with the identity $(1,0)$.
\begin{proposition}\label{w-f-n}
Let $D$ be a coherent domain and $K$ its quotient field. Set $R=D(+)K$, then $R$ is a semi-regular coherent ring. Moreover, $R$ is a coherent ring if and only if $D$ is a field.
\end{proposition}

\begin{proof} Following from \cite[Remark 1]{FA05}, $R$ is a strongly $\phi$-ring, so $\Nil(R)=0(+)K$ and any ideal of $R$ can compare with $\Nil(R)$. Then, by \cite[Corollary 3.4]{AW09}, any ideal of $R$ is of the form $I(+)K$ and $0(+)L$, where $I$ is a nonzero ideal of $D$ and  $L$ is a $D$-submodule of $K$.
If $I$ is a nonzero ideal of $D$, then $I(+)K$ is a semi-regular ideal of $R$ obviously. Let $I(+)K$ be  finitely generated $R$-ideal  generated by $\{(d_1,x_1),...,(d_n,x_n)\}$. Then it is easy to verify $I$ is generated by $\{d_1,...,d_n\}$. Since $D$ is a coherent ring, then there is a short exact sequence $D^m\rightarrow D^n\rightarrow I\rightarrow 0$ of $D$-modules. Since $R$ is a flat $D$-module, we obtain an exact sequence  $R^m\rightarrow R^n\rightarrow I(+)K\rightarrow 0$ of $D$-modules by tensoring $R$. It is easy to verify this is also an $R$-exact sequence. So $I(+)K$ is a finitely presented $R$-ideal. Since any element in the finitely generated $R$-ideal $0(+)L$ is nilpotent, the ideal of the form $0(+)L$  is not semi-regular. Thus $R$ is a semi-regular coherent ring.

Obviously, if $D$ is a field, then $R$ is a coherent ring. Next we will show that if $D$ is not a field, then $R$ is not coherent. Since $(0,1)R$ is a finitely generated $R$-ideal. Consider the natural short exact sequence $0\rightarrow L\rightarrow R\rightarrow (0,1)R\rightarrow 0$, we have  $L=\Nil(R)=0(+)K$. Since $D$ is not a field, then $K$ is not finitely generated over $D$. By \cite[Lemma 2.2]{B03},  $\Nil(R)$ is not a finitely generated $R$-ideal. Consequently, $(0,1)R$ is not finitely presented. So $R$ is not coherent.
\end{proof}

\begin{theorem}\label{s-r-chase}
Let $R$ be a ring. Then the following statements are equivalent:
\begin{enumerate}
 \item $R$ is semi-regular coherent;
 \item  any direct product of semi-regular flat modules is semi-regular flat;
 \item any direct product of flat modules is semi-regular flat;
 \item any direct product of $R$ is semi-regular flat;

 \item any direct limit of  semi-regular coflat module is semi-regular coflat;
 \item any quotient of semi-regular coflat modules is semi-regular coflat;
 \item the class of semi-regular coflat modules is precovering;
 \item the class of  semi-regular coflat modules is covering;

 \item $\Hom_R(N,E)$ is semi-regular flat module for any semi-regular coflat module $N$ and any injective module $E$;
 \item if $E$ is injective cogenerator, then $\Hom_R(N,E)$ is semi-regular flat  for any semi-regular coflat module $N$;
 \item $\Hom_R(\Hom_R(M,E_1),E_2)$ is semi-regular flat for any semi-regular flat module $M$ and any injective modules $E_1$ and $E_2$;
 \item if $E_1$ and $E_2$ are injective cogenerators, then $\Hom_R(\Hom_R(M,E_1),E_2)$ is semi-regular flat  for any semi-regular flat module $M$ .
\end{enumerate}
\end{theorem}
\begin{proof} $(2)\Rightarrow (3)\Rightarrow (4)$: Trivial.

$(1)\Rightarrow (2)$: Let $I$ be a finitely generated semi-regular ideal of $R$  and $\{F_i\}_{i\in I}$  a family of semi-regular flat modules. Consider the following commutative diagram:
$$\xymatrix{
 I\otimes \prod_{i\in I}F_i \ar[d]_{\phi_I}\ar[r]^{\sigma} & I \prod_{i\in I}F_i \ar[d]_{}\\
 \prod_{i\in I}I\otimes F_i \ar[r]^{\cong} &\prod_{i\in I}IF_i.
}$$
Since $R$ is a semi-regular coherent ring, $I$ is finitely presented ideal. Then $\phi_I$ is an isomorphism, thus the epimorphism $\sigma$ is actually an isomorphism. Consequently, $\prod_{i\in I}F_i$ is semi-regular flat.

$(4)\Rightarrow (1)$: Let $I$ be a finitely generated semi-regular ideal of $R$. Consider the following commutative diagram:
$$\xymatrix{
 I\otimes \prod_{i\in I}R\ar[d]_{\phi_I}\ar[r]^{\sigma} & I \prod_{i\in I}R \ar[d]_{f}\\
 \prod_{i\in I}I\otimes R \ar[r]^{\cong} &\prod_{i\in I}I.
}$$
Since $I$ is a finitely generated ideal, then $f$ is an isomorphism. Since $\prod_{i\in I}R$ is a semi-regular flat module, the epimorphism  $\sigma: I\otimes \prod_{i\in I}R\rightarrow I\prod_{i\in I}R$ is an isomorphism. So $\phi_I$ is an isomorphism, thus $I$ is  finitely presented (See \cite[Theorem 2]{ELMUT1969}).

$(1)\Rightarrow (5)$: Let $I$ be a finitely generated semi-regular ideal of $R$ and $\{M_i\}_{i\in I}$  a direct system of  semi-regular coflat modules. Then $\lim\limits_{\longrightarrow }\Ext^1_R(R/I,M_i)=0$. Consider the short exact sequence $0\rightarrow I\rightarrow R\rightarrow R/I\rightarrow 0$, we have the following commutative diagram with rows exact:
$$\xymatrix{
 \ar[r]^{} &\lim\limits_{\longrightarrow } \Hom_R(R,M_i) \ar[d]_{\varphi_R}\ar[r]^{} & \lim\limits_{\longrightarrow }\Hom_R(I,M_i) \ar[r]^{}\ar[d]^{\varphi_I}& \lim\limits_{\longrightarrow }\Ext^1_R(R/I,M_i) \ar[r]^{}\ar[d]^{\varphi^1_{R/I}}&0 \\
 \ar[r]^{} &\Hom_R(R,\lim\limits_{\longrightarrow } M_i) \ar[r]^{} &\Hom_R(I,\lim\limits_{\longrightarrow } M_i)\ar[r]^{} & \Ext^1_R(R/I,\lim\limits_{\longrightarrow } M_i) \ar[r]^{} &0.
}$$
Since $R$ is a semi-regular coherent ring, $I$ is a finitely presented ideal, then $\varphi_{I}$ is an isomorphism. Since $\varphi_{R}$ is an isomorphism, then $\varphi^1_{R/I}$ is also an isomorphism. Consequently, $\lim\limits_{\longrightarrow } M_i$ is a semi-regular coflat module.

$(5)\Rightarrow (1)$: Let $I$ be a finitely generated semi-regular ideal, $\{M_i\}_{i\in I}$ a direct limit of  $R$-modules. Suppose $\alpha: I\rightarrow \lim\limits_{\longrightarrow }M_i$ is an $R$-homomorphism. For any $i\in I$, $E(M_i)$ is the injective  envelope of $M_i$, then $E(M_i)$ is a semi-regular coflat module. By (5), $\alpha$ can be extended to  be $\beta:R\rightarrow \lim\limits_{\longrightarrow }E(M_i)$. So there exists $j\in I$ such that $\beta$ can factor through $R\rightarrow E(M_j)$. Since the composition $I\rightarrow R\rightarrow E(M_j)\rightarrow E(M_j)/M_j$ becomes to be $0$ in the direct limit. We can assume  $I\rightarrow R\rightarrow E(M_j)$ factors through  $M_j$. Then $\alpha$ factor through $M_j$. So the natural epimorphism $\lim\limits_{\longrightarrow } \Hom_R(I,M_i)\rightarrow \Hom_R(I, \lim\limits_{\longrightarrow }M_i)$. So $I$ is a finitely presented ideal.

$(5)\Leftrightarrow (6)\Leftrightarrow (7)\Leftrightarrow (8)$: Follows from \cite[Lemma 3.4]{zxl20}.

$(1)\Rightarrow (9)$: Let $I$ be a finitely generated semi-regular ideal of $R$,  $E$  an injective module and $N$  a semi-regular coflat module. Consider the short exact sequence $0\rightarrow I\rightarrow R\rightarrow R/I\rightarrow 0$. Then we have the following commutative diagram with rows exact (Use $(-,-)$ to instead of $\Hom_R(-,-)$)
$$\xymatrix{
 0 \ar[r]^{} &\Tor_1^R(R/I,(N,E))\ar[d]_{\psi^1_{R/I}}\ar[r]^{} & I\otimes(N,E) \ar[r]^{}\ar[d]^{\psi_{I}}& R\otimes(N,E) \ar[r]^{}\ar[d]^{\psi_{R}}& R/I\otimes(N,E) \ar[r]^{}\ar[d]^{\psi_{R/I}}&0\\
 0 \ar[r]^{} &(\Ext_R^1(R/I,N),E) \ar[r]^{} &(I,N),E)\ar[r]^{} &(R,N),E) \ar[r]^{} &(R/I,N),E)\ar[r]^{} &0.
}$$
Since $R$ is a semi-regular coherent ring, then $I$ is finitely presented ideal. Thus $\psi_{R}$ and  $\psi_{I}$ are all isomorphisms, so $\psi^1_{R/I}$ is also an isomorphism. Since $N$ is a semi-regular coflat module, $\Ext_R^1(R/I,N)=0$. So $\Tor_1^R(R/I,\Hom_R(N,E))=0$, and thus $\Hom_R(N,E)$ is a semi-regular flat module.

$(9)\Rightarrow (10)$ and $(11)\Rightarrow (12)$: Trivial.

$(9)\Leftrightarrow (11)$ and $(10)\Leftrightarrow (12)$: Follow by Theorem \ref{101}.

$(10)\Rightarrow (1)$: Let $I$ be a finitely generated semi-regular ideal of $R$,  $N$  a semi-regular coflat module and $E$ an injective cogenerator. Then we have the following commutative diagram with rows exact:
$$\xymatrix{
 I\otimes\Hom_R(N,E)\ar[d]_{f}\ar[r]^{\psi_I} &\Hom_R(\Hom_R(I,N),E) \ar[d]_{g}\\
 R\otimes\Hom_R(N,E) \ar[r]^{\psi_R}_{\cong} &\Hom_R(\Hom_R(R,N),E).
}$$
Note that $f$ is a monomorphism by (10). So $\psi_I$ is a monomorphism. Thus, by {\cite[Propositon  8.14(1)]{hh}}, we have $I$ is an $\{R\}$-Mittag-Leffler module. Since $I$ is a finitely generated ideal,  $I$ is  finitely presented by \cite[Theorem 2]{ELMUT1969}.
\end{proof}

\begin{lemma}\label{207}
Let $R$ be a ring. Then $R$ is a semi-regular coherent ring if and only if the class $\mathcal{F}_{sr}$ of all semi-regular flat modules  is preenveloping.
\end{lemma}

\begin{proof} Suppose $R$ is a semi-regular coherent ring, then
$\mathcal{F}_{sr}$ is closed under  direct products. Since $\mathcal{F}_{sr}$ is closed under pure submodules, the class of semi-regular flat modules is preenveloping by \cite[Corollary 6.2.2, Lemma 5.3.12]{EJ00}. On the other hand, Suppose $\{F_i\}_{i\in I}$ is semi-regular flat modules. Suppose $\prod_{i\in I} F_i\rightarrow F$ is a semi-regular flat preenvelope, then there is a factorization $\prod_{i\in I} F_i\rightarrow F\rightarrow F_i$ for any  $i\in I$. So $\prod_{i\in I} F_i\rightarrow F\rightarrow \prod_{i\in I} F_i$ is an identity map. Thus $\prod_{i\in I} F_i$ is a direct summand of $F$. Hence $\prod_{i\in I} F_i$ is semi-regular flat. Consequently, $R$ is semi-regular coherent ring by Theorem \ref{s-r-chase}.
\end{proof}
\begin{corollary}\label{208}
Suppose ring $R$ is a semi-regular coherent ring and the class $\mathcal{F}_{sr}$ of all semi-regular flat modules is closed under inverse limits, then $\mathcal{F}_{sr}$ is enveloping.
\end{corollary}
\begin{proof}
By Lemma  \ref{207}, the class $\mathcal{F}_{sr}$ is preenveloping. Then $\mathcal{F}_{sr}$ is enveloping by \cite[Corollary 6.3.5]{EJ00}.
\end{proof}

\begin{theorem}\label{213}
 Let $R$ be a ring. Then the following statements are equivalent:
\begin{enumerate}
 \item $R$ is a strong \Prufer\ ring;
 \item any $R$-module has an epimorphic semi-regular flat envelope;
 \item any $R$-module has an  epimorphic semi-regular flat preenvelope.
\end{enumerate}
\end{theorem}
\begin{proof}
$(2)\Rightarrow (3)$: Trivial.

$(3)\Rightarrow (1)$: Let $F$ be a semi-regular flat module, $i:K\rightarrowtail F$  an embedding map and $f:K\twoheadrightarrow F'$  a semi-regular flat preenvelope. Then there is  a homomorphism $g:F'\rightarrow F$ such that $i=gf$. Thus $f$ is a monomorphism. So $K\cong F'$ is a semi-regular flat module. Hence $R$ is a strong \Prufer\ ring by Theorem \ref{203}.

$(1)\Rightarrow (2)$: Suppose $R$ is a strong \Prufer\ ring. Then $R$ is a semi-regular coherent ring by Proposition \ref{sp-sc}. Hence the class $\mathcal{F}_{sr}$  of semi-regular flat modules  is preenveloping by Lemma \ref{207}. So $\mathcal{F}_{sr}$ is closed under submodules by  Theorem \ref{203}. Suppose $\{F_i|i\in I\}$ is a family of semi-regular flat modules.  Since $R$ is a semi-regular coherent ring,  $\prod_{i\in I} F_i$ is semi-regular flat by Proposition \ref{sp-sc}. Since any inverse limit  is a submodule of direct products, semi-regular flat module is closed under inverse limit. So $\mathcal{F}_{sr}$ is enveloping by Corollary \ref{208}.

 Claim that the $\mathcal{F}_{sr}$-envelope of any $R$-module is an epimorphism.  Indeed, let $f: M\rightarrow F$ is an $\mathcal{F}_{sr}$-envelope.
Consider the standard factorization of  $f$ as $f=h\circ g$, where $g:M\twoheadrightarrow \im{f}$ is an epimorphism, $f:\im{f}\rightarrowtail F$ is a monomorphism. We will show $g$ is  an $\mathcal{F}_{sr}$-envelope of $M$, so $f$ is an epimorphism. First, we will show $g$ is an $\mathcal{F}_{sr}$-preenvelope of $M$. Since $\im{f}\in \mathcal{F}_{sr}$ and for any $f': M\rightarrow F'$ with $F'\in \mathcal{F}_{sr}$, there exists a homomorphism $l:F\rightarrow F'$ such that  $l\circ f=f'$, so $l\circ h\circ g=f'$. Next, we will show $g$ is an $\mathcal{F}_{sr}$-envelope of $M$. Consider the following commutative diagram:
$$\xymatrix@R=25pt@C=25pt{
 M \ar@{->>}[r]^{g}\ar@{=}[d] & \im f \ar[d]_{a}\ar@{>->}[r]^{h} & F\ar@{.>}[d]_{b} \\
 M \ar@{->>}[r]^{g} &\im f \ar@{>->}[r]^{h} & F\\
}$$
Suppose $a:\im{f}\rightarrow\im{f}$ is a homomorphism satisfying  $g=a\circ g$. Then $a$ is an epimorphism. Since $f=h\circ g$ is an
$\mathcal{F}_{sr}$-envelope, there is  an isomorphism $b:F\rightarrow F$ such that $b\circ f=b\circ h\circ g=h\circ a\circ g=h\circ g=f$. So $h\circ a=b\circ h$ as $g$ is an epimorphism. Thus
$a$ is a monomorphism, and therefore $a$ is actually an isomorphism. Hence $f$ is an epimorphism.
\end{proof}

\end{document}